\documentclass[twoside,english,authoryear]{elsarticle}
\usepackage[T1]{fontenc}
\usepackage[latin9]{inputenc}
\usepackage{verbatim}
\usepackage{units}
\usepackage{amsmath}
\usepackage{amsthm}
\usepackage{amssymb}
\usepackage{nomencl}

\providecommand{\printnomenclature}{\printglossary}
\providecommand{\makenomenclature}{\makeglossary}
\makenomenclature

\makeatletter
\theoremstyle{plain}
\newtheorem{thm}{\protect\theoremname}
\ifx\proof\undefined
\newenvironment{proof}[1][\protect\proofname]{\par
\normalfont\topsep6\p@\@plus6\p@\relax
\trivlist
\itemindent\parindent
\item[\hskip\labelsep\scshape #1]\ignorespaces
}{%
\endtrivlist\@endpefalse
}
\providecommand{\proofname}{Proof}
\fi

\usepackage{caption}

\makeatletter
\def\ps@pprintTitle{%
  \let\@oddhead\@empty
  \let\@evenhead\@empty
  \let\@oddfoot\@empty
  \let\@evenfoot\@oddfoot
}
\makeatother

\usepackage{fullpage}

\makeatother

\usepackage{babel}
\providecommand{\theoremname}{Theorem}

\begin{document}

\begin{frontmatter}{}

\title{A note on Brehm's extension theorem}

\author[TUC]{P.~Osinenko}

\ead{pavel.osinenko@etit.tu-chemnitz.de}

\address[TUC]{Laboratory for Automatic Control and System Dynamics; Technische Universit\"at Chemnitz, 09107 Chemnitz, Germany }

\begin{abstract}
Brehm's extension theorem states that a non--expansive map on a finite
subset of a Euclidean space can be extended to a piecewise--linear
map on the entire space. In this note, it is verified that the proof of
the theorem is constructive provided that the finite subset consists
of points with rational coordinates. Additionally, the initial non--expansive
map needs to send points with rational coordinates to points with
rational coordinates. The two--dimensional case is considered.
\end{abstract}


\end{frontmatter}{}

\section{Introduction}

Brehm's extension theorem is stated as follows:

\begin{thm}
	Let $M$ be a finite subset of $\mathbb{R}^{n}$, and $\varphi:M\rightarrow\mathbb{R}^{m},m\le n$
	a map with the property that for any $a,b\in M$, the condition $\|\varphi(a)-\varphi(b)\|\le\|a-b\|$
	holds. Then, there exists a piecewise--linear map $f:\mathbb{R}^{n}\rightarrow\mathbb{R}^{m}$
	such that $\forall a\in M.f(a)=\varphi(a)$.
\end{thm}

Here, $\|\bullet\|$ denotes the Euclidean distance. A map with the
property described is also called \textbf{non--expansive}. The theorem
was first addressed by \citet{Kirszbraun1934-ext-thm} and \citet{Valentine1943-ext-thm},
and then revisited by \citet{Brehm1981-ext-thm}. A similar proof
can be found in \citet{Akopyan2008-constr-Kirszbraun} and \citet[p.~21]{Petrunin2014-PL-maps}.
In the present work, it is verified that the theorem admits a constructive
proof in the sense of Bishop's constructive mathematics \citep{Bishop1985-constr-analysis}
provided that $M$ and $\varphi(M)$ consist of points with rational
coordinates. Only the planar case $m=n=2$ is considered.

\section{Preliminaries}

In this section, selected basics of constructive mathematics are briefly discussed. For
a comprehensive description, refer, for example, to \citep{Bishop1985-constr-analysis,Bridges1987-varieties,Bridges2007-techniques,Ye2011-SF,Schwichtenberg2012-constr-analysis-witnesses}.
Bishop's constructive mathematics uses the notion of an \textbf{operation}
which is an algorithm that produces a unique result in a finite number
of steps for each input in its domain. For example, a \textbf{real}
number $x$ is a \emph{regular} Cauchy sequence of rational numbers
in the sense that 
\[
	\forall n,m\in\mathbb{N}.|x(n)-x(m)|\leq\frac{1}{n}+\frac{1}{m}
\]
where $x(n)$ is an \emph{operation} that produces the $n$th rational
approximation to $x$. A \textbf{set} is a pair of operations: $\in$
determines that a given object is a member of the set, and $=$ determines
whenever two given set members are equal. Existence and universal
quantifiers are interpreted as follows: $\exists x\in A.\varphi\left[x\right]$
means that an operation has been derived that constructs an instance
$x$ along with a proof of $x\in A$ and a proof of the logical formula
$\varphi\left[x\right]$ as \emph{witnesses}; $\forall x\in A.\varphi\left[x\right]$
means that an operation has been derived that proves $\varphi\left[x\right]$
for any $x$ provided with a witness for $x\in A$. A set $A$ is
called \textbf{inhabited }if there exists an $x\in A$. A \textbf{finite
set} is a set that admits a bijection to a set $\left\{ 1,2,...n\right\} $
for some $n\in\mathbb{N}$ which means that all its elements are enumerable.
The Euclidean space $\mathbb{R}^{n}$ is a normed space with the norm
$\|x\|\triangleq\left(\sum_{i=1}^{n} x_i^{2} \right)^{\nicefrac{1}{2}}$
where $x_i$ is the $i$--th coordinate of $x$. The metric is defined
as $\|x-y\|$ for any $x,y\in\mathbb{R}^{n}$. A point $x$ in the
Euclidean space is called \textbf{algebraic} if its coordinates are
algebraic numbers.

A (closed) \textbf{polytope }is a union of polyhedrons\textbf{ }whereas
a (closed) \textbf{polyhedron }is an inhabited set of points of the
Euclidean space satisfying linear inequalities $Ax\le b,A\in\mathbb{R}^{n\times n},b\in\mathbb{R}^{n\times1}$.
If the entries of $A$ and $b$ are solely algebraic numbers, then
the polyhedron is called \emph{algebraic}. If a polytope is a union
of solely algebraic polyhedrons, then it is algebraic as well. A polyhedron
(respectively, polytope) $P$ is \textbf{bounded} if there exists
a rational number $\bar{x}$ such that $\|x\|\le\bar{x}$ for any
$x$ in $P$. An $n$--dimensional \textbf{simplex} is a convex hull
of $n+1$ affinely independent points. A \textbf{triangulation} of
a bounded algebraic $n$--dimensional polyhedron $P$ is a finite
set $\left\{ T_{i}\right\} _{i}$ of algebraic simplices whose intersections
are at most $(n-1)$--dimensional and such that $P=\cup_{i}T_{i}$.
For example, a triangulation of a two--dimensional polyhedron is a
collection of non--degenerate triangles that may have a common vertex
or segment of an edge, but no two--dimensional intersection.

A \textbf{motion} in the plane $\mathbb{R}^{2}$ is a map $f$ that
is a composition of a translate, a rotation and a reflection. Clearly,
it is distance--preserving in the sense that $\|f(x)-f(y)\|=\|x-y\|$
for any $x,y$. Translate, rotation and reflection can be described
as linear transformations of the form $x\mapsto Tx$ where $T$ is
the transformation matrix. A motion can be thus described by a transformation
matrix as well. Notice that a motion is always invertible since the
corresponding transformation matrix is regular. A motion is called
\emph{algebraic} if the corresponding transformation matrix comprises
solely of algebraic entries. For example,
\[
	f(x)=\left[\begin{array}{cc}
	\sqrt{1-\frac{9}{25}} & \frac{3}{5}\\
	-\frac{3}{5} & \sqrt{1-\frac{9}{25}}
	\end{array}\right] x
\]
is an algebraic motion and describes the clockwise rotation by the
angle $\arcsin\frac{3}{5}$. In contrast, $f(x)=x+\pi$ is not an
algebraic motion. Any three algebraic points forming a non--degenerate
triangle can be moved by an algebraic motion to new algebraic points
preserving the respective distances \citep{Petrunin2014-PL-maps}.
An \textbf{algebraic piecewise--linear map} $f$ on a bounded algebraic
two--dimensional polyhedron $P$ is a pair of a triangulation $\left\{ T_{i}\right\} _{i}$
of the polyhedron and a collection of algebraic motions $\left\{ f_{i}\right\} _{i}$
such that $f|_{T_{i}}=f_{i}$. For example, folding of a piece of
paper without ripping can be considered as an algebraic piecewise--linear
map if foldings are performed at algebraic points. Notice that each
algebraic piecewise--linear map has a triangulation and a collection
of algebraic motions as witnesses. Clearly, an algebraic piecewise--linear
map is non--expansive. The notion of an algebraic piecewise--linear
map can be directly generalized to algebraic polytopes.

It is important to notice that, for arbitrary real numbers $x,y$,
it is not decidable whether $x=y$ or $x\ne y$. This limitation has
a number of consequences for the theory of the Euclidean space $\mathbb{R}^{n}$.
In particular, no full power of set operations is available. For
example, if $A$ and $B$ are arbitrary sets in $\mathbb{R}^{n}$,
it is not decidable whether $A\cap B=\emptyset$ or $A\cap B$ is inhabited.
In this note, set operations are limited to the class of sets of the form $\left\{ x:\bigvee_{i=1}^{N}\bigwedge_{j=1}^{M_{i}}\mathcal{E}_{ij}\right\} $
with $\mathcal{E}_{ij}$ being a formula of the type $A_{ij}x\bullet b_{ij}$
or $\|f_{ij}(x)\|\bullet\|g_{ij}(x)\|$ where ``$\bullet$'' denotes
``$<$'',''$\le$'' or ''$=$'' and $f_{ij}$ and $g_{ij}$
are algebraic piecewise--linear maps on algebraic polytopes. Denote this class by $\mathcal{AS}$. For example, an algebraic polytope
itself belongs to $\mathcal{AS}$. Further, the set complement of
a set $A\in\mathcal{AS}$, denoted by $\mathbb{R}^{n}\backslash A$
is, again, an element of the class $\mathcal{AS}$ (it can be done
by transforming the sign ``$<$'',''$\le$'' or ''$=$'' in
the respective formula). Notice that if $f_{ij}$ and $g_{ij}$ are
algebraic motions, each $\|f_{ij}(x)\|\le\|g_{ij}(x)\|$ is equivalent
to $\sum_{k=1}^{n}(f_{ij}(x))_{k}^{2}\le\sum_{k=1}^{n} (g_{ij}(x))^2_{k}$. The same applies
if $f_{ij}$ and $g_{ij}$ are algebraic piecewise--linear maps by
considering the inequalities on the simplices where $f_{ij}$ and
$g_{ij}$ are both algebraic motions. If a set $A$ has the form $\left\{ x:\bigwedge_{i=1}^{N}\|f_{i}(x)\|<\|g_{i}(x)\|\right\} $
with $f_{i}$ and $g_{i}$ algebraic piecewise--linear, then
its boundary is defined to be the set $\partial A\triangleq\left\{ x:\bigwedge_{i=1}^{N}\|f_{i}(x)\|=\|g_{i}(x)\|\right\} $
which in turn belongs to $\mathcal{AS}$. Lemma 4.1 from \citet[p.~8]{Beeson1980} states decidability of equality over
the field of algebraic real numbers. This allows performing the ordinary set operations
on the sets of the described class. In the following, the extension theorem is revisited and
verified to admit a constructive proof.

\section{Extension theorem}

The proof of the following theorem is mostly based on \citep{Akopyan2008-constr-Kirszbraun,Petrunin2014-PL-maps}.

\begin{thm}
	Let $\left\{ a_{i}\right\} _{i=1}^{n},\left\{ b_{i}\right\} _{i=1}^{n}$
	be finite subsets of points in $\mathbb{R}^{2}$ with rational coordinates
	such that $\forall i,j.\|b_{i}-b_{j}\|\le\|a_{i}-a_{j}\|$. Let $A$
	be the convex hull of $\left\{ a_{i}\right\} _{i=1}^{n}$. Then, there
	exists an algebraic piecewise--linear map $f:A\rightarrow\mathbb{R}^{2}$
	such that $\forall i.f(a_{i})=b_{i}$.
\end{thm}

\begin{proof}
	The theorem is proven by induction on the number of points. If $n=1$,
	one may take $f(x):=x+(b_{1}-a_{1})$ which is clearly an algebraic
	motion on the entire space. Suppose that an algebraic piecewise--linear
	map $g:A\rightarrow\mathbb{R}^{2}$, such that $\forall i=1,\dots,n-1.g(a_{i})=b_{i}$,
	was constructed. Define a set $\Omega:=\left\{ x:x\in A\wedge\|a_{n}-x\|<\|b_{n}-g(x)\|\right\} $.
	Since $g$ is algebraic, it is decidable whether $b_{n}=g(a_{n})$ or
	$b_{n}\ne g(a_{n})$. In the former case, take $f$ to be $g$.
	In the latter, $\Omega$ is inhabited since $a_{n}\in\Omega$. Notice that if
	$x$ belongs to $\Omega$, then so does the line segment
	between $a_{n}$ and $x$. Take a point $y$ in this line segment.
	Then, 
	\[
		\|a_{n}-y\|+\|y-x\|=\|a_{n}-x\|.
	\]
	Since $x\in\Omega$,
	\[
		\|a_{n}-x\|\le\|b_{n}-g(x)\|.
	\]
	The map $g$ is an algebraic piecewise--linear which implies
	\[
		\|g(x)-g(y)\|\le\|x-y\|.
	\]
	Therefore,
	\begin{equation}
		\begin{aligned}
			\|a_{n}-y\| & =\|a_{n}-x\|-\|y-x\|\\
			 & <\|b_{n}-g(x)\|-\|g(x)-g(y)\|\\
			 & \le\|b_{n}-g(y)\|
		\end{aligned}
	\label{eq:star-shaped-property}
	\end{equation}
	where the last line follows from the triangle inequality $\|g(x)-g(y)\|\le\|b_{n}-g(y)\|+\|b_{n}-g(y)\|$.
	Since $\|a_{n}-y\|<\|b_{n}-g(y)\|$, it follows that $y\in\Omega$.
	Now, the boundary $\partial_{A}\Omega:=\partial\Omega\cap A$
	is inspected. Let $\left\{ T_{i}\right\} _{i}$ be the triangulation
	of $A$ such that $g$ on each triangle is a motion $g_{i}$. Let
	$c_{n}:=g_{i}^{-1}(a_{n})$. Notice that $c_{n}$ is an algebraic
	point. Since $g_{i}$ is a motion and $g|_{T_{i}}=g_{i}$, for any
	$x\in T_{i}$, it follows that%
	%
	\begin{equation}
		\|c_{n}-x\|=\|b_{n}-g(x)\|.
		\label{eq:motion-on-triangle}
	\end{equation}
	Since $\Omega$ and $T_{i}$ belong to $\mathcal{AS}$, it is decidable
	whether the intersection $\Omega\cap T_{i}$ is inhabited. Suppose
	it is inhabited. Then, consider the line
	\[
		l_{i}:=\left\{ x:\|x-a_{n}\|=\|x-c_{n}\|\right\} .
	\]
	It follows that:
	\begin{equation}
		\partial_{A}\Omega\cap T_{i}=\left\{ x:\|a_{n}-x\|=\|b_{n}-g(x)\|\land x\in T_{i}\land x\in A\right\} \label{eq:omega-boundary-in-triangle}
	\end{equation}
	and 
	\begin{equation}
		l_{i}\cap T_{i}\cap A=\left\{ x:\|x-a_{n}\|=\|x-c_{n}\|\land x\in T_{i}\land x\in A\right\}.
		\label{eq:line-in-triangle}
	\end{equation}
	Matching \eqref{eq:omega-boundary-in-triangle} with \eqref{eq:line-in-triangle}
	using \eqref{eq:motion-on-triangle}, one can see that $\partial_{A}\Omega\cap T_{i}$
	is a line segment. Since $\left\{ T_{i}\right\} _{i}$ is a finite
	set, $\partial_{A}\Omega$ is a finite collection of line segments.
	Consider a line segment $\omega_{i}$ of $\partial_{A}\Omega$. Let
	$\tau_{i}$ be the triangle formed by $a_{n}$ and $\omega_{i}$.
	Let $f_{i}$ be an algebraic motion that maps $a_{n}$ to $b_{n}$
	and the endpoints of $\omega_{i}$ to their respective positions under
	$g_{i}$. For $x\in\omega_{i}$, it follows that $g(x)=g_{i}(x)$ and so $g(x)=f_{i}(x)$.
	Let $f|_{\tau_{i}}:=f_{i}$ and $f|_{A\backslash\Omega}:=g$. Further,
	since $\partial\Omega,\partial_{A}\Omega\in\mathcal{AS}$, it is decidable
	whether $\Delta:=\partial\Omega\cap\partial_{A}\Omega$ is
	inhabited. If this is the case (for otherwise, the result is trivial), consider
	the algebraic polytopes $D_{k},k=1,\dots,m$ formed by the endpoints
	of $\Delta$ that lie on $\partial\Omega\cap A$ and the line segments
	from these endpoints to $a_{n}$. Let $\lambda_{1}$ and $\lambda_{2}$
	denote the said endpoints for some algebraic polytope $D_{k}$. Since
	$f$ coincides with $g$ on the line segments $\left[a_{n},\lambda_{1}\right]$
	and $\left[a_{n},\lambda_{2}\right]$, and, moreover, it acts as algebraic
	motions on these line segments, and since $g$ is non--expansive,
	it follows that
	\[
		\begin{aligned}
			\|\lambda_{1}-a_{n}\|= & \|g(\lambda_{1})-b_{n}\|,\\
			\|\lambda_{2}-a_{n}\|= & \|g(\lambda_{2})-b_{n}\|\\
			\|g(\lambda_{1})-g(\lambda_{2})\|\le & \|\lambda_{1}-\lambda_{2}\|.
		\end{aligned}
	\]
	The required map on $D_{k}$ can be constructed as follows. Translate
	and rotate $D_{k}$ so that $\left[a_{n},\lambda_{1}\right]$ coincides
	with $\left[b_{n},g(\lambda_{1})\right]$. This can be done since
	the initial and the new vertices of $D_{k}$ are algebraic. So far,
	the line segment $\left[a_{n},\lambda_{2}\right]$ ``turned'' around
	$b_{n}$ closer to $\left[a_{n},\lambda_{1}\right]$. Draw a line
	segment from $g(\lambda_{1})$ to $g(\lambda_{2})$ which is the chord
	of the circle on which the point $\lambda_{2}$ slid to the new position.
	Take the middle point of the chord and fold $D_{k}$ around the ray
	going from $a_{n}$ to this middle point so that $\lambda_{2}$ matches
	with $g(\lambda_{2})$. The resulting map is thus constructed by translating
	and rotating the whole $D_{k}$ and then reflecting the fragment---to--fold
	around the said ray which constitutes a piecewise--linear map. This
	map is clearly algebraic since all the points involved are algebraic.
\end{proof}

\appendix

\bibliographystyle{apalike}

\end{document}